\documentclass[12pt,reqno]{article}

\usepackage[usenames]{color}
\usepackage{amssymb}
\usepackage{graphicx}
\usepackage{amscd}

\usepackage[colorlinks=true,
linkcolor=webgreen,
filecolor=webbrown,
citecolor=webgreen]{hyperref}

\definecolor{webgreen}{rgb}{0,.5,0}
\definecolor{webbrown}{rgb}{.6,0,0}

\usepackage{color}
\usepackage{fullpage}
\usepackage{float}

\usepackage{graphics,amsmath,amssymb}
\usepackage{amsthm}
\usepackage{amsfonts}
\usepackage{latexsym}
\usepackage{epsf}

\usepackage{verbatim}
\usepackage{enumerate}
\usepackage{mathtools}

\usepackage{tikz}
\usepackage{tcolorbox}

\usepackage[blocks]{authblk}

\usepackage{algpseudocode}

\DeclarePairedDelimiter\floor{\lfloor}{\rfloor}

\theoremstyle{plain}
\newtheorem{theorem}{Theorem}
\newtheorem{corollary}[theorem]{Corollary}

\theoremstyle{definition}

\newtheorem{example}{Example}

\theoremstyle{remark}
\newtheorem{remark}{Remark}

\begin{document}

\title{Card Tricks and Information}
\author{Aria Chen}
\author{Tyler Cummins}
\author{Rishi De Francesco}
\author{Jate Greene}
\author{Alexander Meng}
\author{Tanish Parida}
\author{Anirudh Pulugurtha}
\author{Anand Swaroop}
\author{Samuel Tsui}
\affil{PRIMES STEP}
\author{Tanya Khovanova}
\affil{MIT}
\date{}

\maketitle

\begin{abstract}
Fitch Cheney's 5-card trick was introduced in 1950. In 2013, Mulcahy invented a 4-card trick in which the cards are allowed to be displayed face down. We suggest our own invention: a 3-card trick in which the cards can be face down and also allowed to be placed both vertically and horizontally. We discuss the theory behind all the tricks and estimate the maximum deck size given the number of chosen cards. We also discuss the cases of hiding several cards and the deck that has duplicates.
\end{abstract}

\section{Introduction}

The five-card trick appeared in 1950 in \textit{Math Miracles}, \cite{Lee}, where it was attributed to Fitch Cheney. It was popularized in 2002 by Kleber \cite{K}. Throughout this paper, we assume that a standard deck of cards contains 52 cards. This is how the trick is seen by the audience.

\begin{tcolorbox}
\textbf{The 5-card trick.} The audience gives 5 cards from a standard deck to the magician's assistant. The assistant hides one card, then arranges the remaining four cards in a row, face up. The magician uses the description of how the cards are placed in a row to guess the hidden card.
\end{tcolorbox}

To avoid other means of signaling information, the magician can be in a different room, and an audience member can call the magician and describe how the cards are arranged. We describe the details of how this trick is performed in Section~\ref{sec:cm}. 

In 2002, Kleber \cite{K} found a different way to perform this trick, which allowed for a much larger deck. The trick could be performed with a deck of 124 cards, which is the largest possible. More generally, Kleber proved that if the audience chooses $K$ cards instead of 5, the largest possible deck size $N$ is $K!+K-1$. We reproduce the proof in Section~\ref{sec:fu}.

The follow-up four-card trick appeared in Colm Mulcahy's book \textit{Mathematical Card Magic: Fifty-Two New Effects} in 2013 \cite{Mulcahy}. The big idea is the same, but the trick uses fewer cards. This is possible due to extra flexibility: the assistant is allowed to put some of the cards face down.

\begin{tcolorbox}
\textbf{The 4-card trick.} The audience gives 4 cards from a standard deck to the magician's assistant. The assistant hides one card and then arranges the remaining three cards in a row, face up or down. The magician uses the description of how the cards are placed in a row to guess the hidden card.
\end{tcolorbox}

We describe the details of how this trick is performed in Section~\ref{sec:cm}.

We start the paper with the case when all the cards are face up in Section~\ref{sec:fu}. We start with definitions and consider two versions of the trick. In the first, simpler one, the audience chooses which card to hide. In the second version, the assistant chooses which card to hide. In this section, we study two scenarios: cards are in a line and in a circle. We also allow each card to be rotated in $R$ different ways. For each version and scenario, we calculate the maximum possible deck size $N$ given the number of cards $K$ the assistant gets and the number of rotations $R$.

In Section~\ref{sec:cm}, we discuss Cheney's method for the case when the assistant chooses the hidden card. We start with definitions and the original trick. Next, we provide a general formula for the deck size depending on the number of cards $K$ the assistant gets. Then, we show how this method can be applied to Mulcahy's trick when the cards are allowed to be face down. We calculate the largest deck size with which Mulcahy's trick can be performed as a function of $K$. Then, we extend the formula to allow each card to be rotated in $R$ ways.

Suppose we allow our cards to be placed vertically or horizontally, corresponding to $R=2$. Then, for $K=3$, the maximum deck is 54, which allows one to perform the trick with the standard deck of 52 cards. Moreover, it is possible to add two jokers to the deck. We describe this new 3-card trick for the standard deck in detail in Section~\ref{sec:trick} to allow it to be performed by magicians.

In Section~\ref{sec:rfaud}, we discuss the case when the cards can be rotated and flipped as before. But now, the audience chooses the hidden card. We calculate the upper bound on the deck size given $K$ and the number of allowed rotations $R$.

In Section~\ref{sec:duplicates}, our deck has each card duplicated. We assume that the cards are always face up. We calculate the maximum deck size when the audience chooses the hidden card. When the assistant chooses the hidden card, we estimate the upper bound for $K < 6$. We provide examples that match the bound for $K =2$ and 3. We also suggest a general strategy that works for any $K$ but doesn't reach the theoretical bound.

In Section~\ref{sec:morethanone}, the magician needs to guess more than one card. We provide details for two hidden cards, where we calculate the maximum deck size for the case when the audience chooses the hidden cards and estimate the bound for when the assistant chooses the hidden cards. We also suggest a strategy that is close to the bound. The strategy is generalizable for any number of hidden cards.

\section{Cards are face up}
\label{sec:fu}

\subsection{Definitions}
Let us assume the deck size is $N$, and the audience chooses $K$ cards to give to the assistant. Given the number of cards and conditions of the trick, we would like to estimate the largest possible deck of cards this trick can be performed on. When displaying the cards, the assistant is also allowed to rotate any of the cards in $R$ different ways. For example, when $R=2$, the cards could be vertical or horizontal.

This trick will look more impressive if the magician is in the room and the cards are asymmetrical. For example, many cheap decks have different widths of the border margin on the left and right sides. Choosing the correct margin is equivalent to choosing a rotation but hidden from the audience.

We consider two different scenarios:
\begin{itemize}
\item The audience chooses the hidden card for the magician to guess.
\item The assistant chooses the hidden card for the magician to guess.
\end{itemize}

We added the case when the audience chooses the hidden card, as it is simpler and provides insight into the other case.

We call each way the assistant can arrange the given cards \textit{a message}. We are also interested in the total number of messages using all possible cards. In most of the methods in this section, the magician and the assistant agree to correspond each arrangement of the cards to a number, which we call the \textit{signaling number} denoted as $S$. The total number of ways to arrange the cards, the \textit{maximum signaling number}, is denoted as $M$. In addition, we can exclude the hidden card from the set of face-up cards on the table. We call such cards \textit{excluded} and denote their number as $E$. We describe the Cheney method separately; in this method, the notion of the signaling number is adjusted.

\subsection{The audience chooses the hidden card}

The magician sees $E$ cards; there are $N-E$ possibilities for the hidden card. For every possible choice of the hidden card, the assistant must send a unique message to the magician. There are $M$ possible messages that could be sent. Thus, $N-E \le M$, implying
\[N\le M+E.\]

This bound is achievable. The magician and assistant agree on the ordering of the cards. They also agree on how an arrangement of the cards on the table can be translated to a number. For example, if the cards are in a line without rotations, then they can use the lexicographic ordering of permutations. The assistant looks at the cards that will be arranged on the table and excludes them from all the cards in a deck. Then, the order of the nonexcluded cards is adjusted to make them consecutive. The assistant needs to signal the index $S$ of the hidden card in the adjusted order. This can be done as there are $M$ ways to do so, and there are $M$ non-excluded cards.

The magician looks at the arrangement of the cards on the table and calculates the signaling number $S$. The magician excludes the cards on the table from the list of all possible hidden cards. There are now $M$ possible hidden cards, which the magician arranges in order in their mind, and announces the $S$-th card as the hidden card. Thus, the bound is achievable, and $N = M+E$.

\begin{tcolorbox}
\textbf{Formula for cards face up, the audience chooses the hidden card.} The maximum deck size is 
\begin{equation}
\label{eq:fu}
N = M +E.
\end{equation}
\end{tcolorbox}

\subsubsection{Cards are in a line}
\label{sec:rotatedAud}

The assistant is given $K$ cards. The audience hides one of the cards, and the assistant orders the other cards. In addition, in this trick, the assistant is also allowed to rotate any of the cards in $R$ different ways. The magician looks at these cards and guesses the hidden card. 

\begin{theorem}
When the cards are in a line and allowed to be rotated, and the audience chooses the hidden card, the maximum deck size is
\[N = R^{K-1}(K-1)! + K - 1.\]
\end{theorem}

\begin{proof}
There are $K-1$ excluded cards. The cards can be permuted, and each can be rotated in $R$ ways. Thus, $E = K-1$ and $M = R^{K-1}(K-1)!$. By Eq.~\ref{eq:fu}, the maximum deck size is
\[N = M + E = R^{K-1}(K-1)! + K - 1.\]
\end{proof}

Table~\ref{tab:rot_aud} shows the maximum deck sizes for small values of $K$ and $R$.

\begin{table}[ht!]
\centering
\begin{tabular}{|c|c|c|c|c|c|c|} 
 \hline
 $R/K$ & 1	& 2	& 3	& 4	& 5	& 6 \\
 \hline
 1	& 1	& 2	& 4	& 9	& 28	& 125\\
 2	& 1	& 3	& 10	& 51	& 388	& 3845\\
 3	& 1	& 4	& 20	& 165	& 1948	& 29165\\
 4	& 1	& 5	& 34	& 387	& 6148	&122885\\
 5	& 1	& 6	& 52	& 753	& 15004 & 375005\\
\hline
\end{tabular}
\caption{Maximum deck sizes for small values of $K$ and $R$ when the cards are in a line, and the audience chooses the hidden card.}
\label{tab:rot_aud}
\end{table}

When $R=1$, the assistant can only permute the available cards corresponding to the first row of Table~\ref{tab:rot_aud}. The maximum size of the deck is 
\[N = M + E = (K-1)! + K-1.\]
The corresponding sequence is A005095 in the On-Line Encyclopedia of Integer Sequences (OEIS) \cite{OEIS}, and starting from $K=1$, we get
\[1,\ 2,\ 4,\ 9,\ 28,\ 125,\ 726,\ 5047,\ 40328,\ 362889,\ 3628810,\ \ldots.\]

\begin{example}
If $R=2$ and $K=4$, the maximum deck size is 51. That means we can perform this trick using the standard deck if we secretly remove one card.
\end{example}

\subsubsection{Cards are in a circle}
\label{sec:circleAud}

In this section, the assistant has to arrange the cards in a circle. When the assistant describes cards to the magician, the assistant starts with a random card in the circle and describes the circle clockwise. For example, we can assume that an audience member chooses the starting point. As everywhere in this section, we assume the cards are face up, and each card can be rotated in $R$ ways.

\begin{theorem}
When the cards are in a circle and are allowed to be rotated, and the audience chooses the hidden card, the maximum deck size is
\[N = R^{K-1}(K-2)!+K-1.\]
\end{theorem}

\begin{proof}
The maximum signaling number $M$ is $R^{K-1}(K-2)!$: the number of permutations up to rotations along the circle times the number of rotations of the cards. The number of excluded cards is $K-1$. Thus, by Eq.~\ref{eq:fu}
\[N = M + E = R^{K-1}(K-2)!+K-1.\]
\end{proof}

Table~\ref{tab:circleaud} shows the maximum deck sizes for small values of $K$ and $R$.
\begin{table}[ht!]
\centering
\begin{tabular}{|c|c|c|c|c|c|} 
 \hline
 $R/K$ & 2 & 3 & 4 & 5 & 6 \\
 \hline
 1	& 2	& 3	& 5 & 10 & 29\\
 2	& 3	& 6	& 19 & 100 & 773\\
 3	& 4	& 11 & 57 & 490	& 5837\\
 4	& 5	& 18 & 131 & 1540 & 24581\\
 5	& 6	& 27 & 253 & 3754 & 75005\\
\hline
\end{tabular}
\caption{The bound for deck sizes for small values of $K$ and $R$ when the cards are in a circle, and the audience chooses the hidden card.}
\label{tab:circleaud}
\end{table}

The sequence for $R=1$ starting from $K = 1$ is
\[1,\ 2,\ 3,\ 5,\ 10,\ 29,\ 126,\ 727,\ 5048,\ 40329,\ 362890,\ 3628811,\ 39916812,\ \ldots.\]
If we remove the first term, we get sequence A213169 from OEIS \cite{OEIS}.

\subsection{The bound for the assistant choosing the hidden card}

\subsubsection{General discussion}

Kleber \cite{K} had the following argument describing the maximum possible deck size. Guessing the hidden card by a magician is equivalent to guessing the set of all the cards the assistant receives. We keep in mind that it is an unordered set. The assistant cannot receive more sets of cards than the number of total messages the magician can see. If there are, one message would have to correspond to multiple sets of cards, making it impossible to guess. The total number of messages the assistant can send is $N \choose K-1$ times the number of ways to arrange the cards, which is $M$, the maximum signaling number. Turning this into an inequality, we get
\[{N\choose K} \leq M {N \choose K -1},\]
implying
\[N-K+1 \leq M K,\]
which is equivalent to 
\begin{equation}
\label{eq:bound}
N \leq KM + E.
\end{equation}

Kleber \cite{K} showed that this bound is achievable, and the strategy is called \textit{the best-trick strategy}. Let us say that the $K$ cards the assistant is given are $c_0 < c_1 < c_2 < c_3 < \cdots < c_{K-1}$. The assistant then hides $c_i$, where $i$ is the sum of the cards mod $K$.

If all the cards other than those visible to the magician are renumbered from 0 to $KM-1$, then the number corresponding to the hidden card is reduced by $i$, equaling $c_i - i$. The sum of the visible cards equals $i - c_i$ modulo $K$. Thus, the magician can find the value of the hidden card modulo $K$ using the sum of the values of the cards showing. This narrows the number of non-excluded cards by a factor of $K$. The assistant can signal $M$ numbers with different arrangements of the visible cards.

\begin{tcolorbox}
\textbf{Formula for cards face up, the assistant chooses the hidden card.} The maximum deck size is
\begin{equation}
\label{eq:besttrick}
N = KM + E.
\end{equation}
\end{tcolorbox}

\subsubsection{Cards are in a line}

The assistant is given $K$ cards. The assistant hides one of the cards and orders the other cards. In addition, in this trick, the assistant is also allowed to rotate any of the cards in $R$ different ways. The magician looks at these cards and guesses the hidden card. 

\begin{theorem}
When the cards are allowed to be rotated, and the assistant chooses the hidden card, the maximum deck size is
\[N = R^{K-1} K! +K-1.\]
\end{theorem}

\begin{proof}
The maximum signaling number $M$ is $R^{K-1}(K-1)!$: the number of permutations of $K-1$ cards, times $R^{K-1}$ to account for possible rotations of each card. The number of excluded cards is $K-1$. The result follows from Eq.~\ref{eq:besttrick}: $N = KM + E$.
\end{proof}

Table~\ref{tab:rot} shows the maximum deck sizes for small values of $K$ and $R$.

\begin{table}[ht!]
\centering
\begin{tabular}{|c|c|c|c|c|c|c|} 
 \hline
 $R/K$ & 1	& 2	& 3	& 4	& 5	& 6 \\
 \hline
 1	& 1	& 3	& 8	& 27	& 124	& 725 \\
 2	& 1	& 5	& 26	& 195	& 1924	& 23045 \\
 3	& 1	& 7	& 56	& 651	& 9724	& 174965 \\
 4	& 1	& 9	& 98	& 1539	& 30724 & 737285 \\
 5	& 1	& 11	& 152	& 3003	& 75004 & 2250005\\
\hline
\end{tabular}
\caption{Maximum deck sizes for small values of $K$ and $R$ when the cards are in a line and the assistant chooses the hidden card.}
\label{tab:rot}
\end{table}

\begin{example}
We can see that for $K = R = 3$, we can perform the trick with the standard deck of cards. For $K = 4$ and $R=1$, the bound is $27$, which is half a standard deck, including the joker. The trick can be performed using, for example, two red suits and a red joker.
\end{example}

For the original trick, when the cards are in a line and are not rotated, we get the formula $N = K! + K - 1$, which we already saw \cite{K}. The deck size as a function of $K$ is sequence A030495 in the OEIS \cite{OEIS} and starts as follows
\[1,\ 3,\ 8,\ 27,\ 124,\ 725,\ 5046,\ 40327,\ \ldots.\]

\subsubsection{Cards are in a circle}

Now, the cards are in a circle, and we are allowed to rotate each card in $R$ ways.

\begin{theorem}
When the cards are in a circle and allowed to be rotated, and the assistant chooses the hidden card, the maximum deck size is
\[N=R^{K-1} K (K-2)! +K-1.\]
\end{theorem}

\begin{proof}
The maximum signaling number $M$ is $R^{K-1}(K-2)!$: the number of permutations of $K-1$ cards divided by $K-1$, to account for the fact that it is not known where the circle starts, times $R^{K-1}$ to account for possible rotations of each card. The number of excluded cards is $K-1$. The result follows from Eq.~\ref{eq:besttrick}: $N = KM + E$.
\end{proof}

Table~\ref{tab:rotcircleass} shows the maximum deck sizes for small values of $K$ and $R$.

\begin{table}[ht!]
\centering
\begin{tabular}{|c|c|c|c|c|c|} 
 \hline
 $R/K$ & 2 & 3 & 4 & 5 & 6 \\
 \hline
 1	& 3	& 5	& 11 & 34 & 149\\
 2	& 5	& 14 & 67 & 484 & 4613\\
 3	& 7	& 29 & 219 & 2434 & 34997\\
 4	& 9	& 50 & 515 & 7684 & 147461\\
 5	& 11 & 77 & 1003 & 18754 & 450005\\
\hline
\end{tabular}
\caption{Maximum deck size for small values of $K$ and $R$ when the cards are in a circle and the assistant chooses the hidden card.}
\label{tab:rotcircleass}
\end{table}

When $R=1$, we get
\[N = K(K-2)! + K - 1.\]
The new sequence A372255 starts from index 2 as:
\[3,\ 5,\ 11,\ 34,\ 149,\ 846,\ 5767,\ 45368,\ 403209,\ 3991690,\ 43545611,\ \ldots.\]

\section{The assistant chooses the hidden card. Cheney's method}
\label{sec:cm}

\subsection{Definitions}

We describe Cheney's method in a general setting. This method assumes that the assistant chooses the hidden card. In Cheney's method, the cards are divided into $K-1$ groups: the largest number of groups that guarantee that two cards among $K$ cards are in the same group. The assistant hides one of these cards and places the other card, which we call the \textit{signaling card}, in a pre-agreed fashion. For example, the signaling card could be the leftmost card. Some advanced magicians vary the placement of the signaling card to make it more difficult for the audience to figure out the trick.

When the magician sees the signaling card, the magician knows which group the hidden card belongs to. 

Similar to before, we call the number of ways to arrange the other cards the \textit{maximum signaling number} and denote it as $M$. The assistant has the flexibility to hide either card from the group. Suppose the cards in a group are numbered 0 through $X$, and the assistant has two cards, $A$ and $B$. Suppose $B-A$ is smaller than $A-B$ modulo $X+1$. Then, the assistant hides $B$ and signals $B-A$ modulo $X+1$. Thus, the maximum size of the group is $2M +1$. The three cards that are not a signaling card can be permuted in 6 different ways to signal a number from 1 to 6. Thus, the number of cards in a group can be up to 13. Amazingly, this is the same as the number of cards in a suit in a standard deck.

\begin{tcolorbox}
\textbf{Formula for cards face up, the audience chooses the hidden card, and Cheney's method is used.} The maximum deck size is 
\begin{equation}
\label{eq:cm}
N = (K-1)(2M+1).
\end{equation}
\end{tcolorbox}

For the original Cheney trick with $K$ cards, the total deck size is
\[N = (K-1)(2M+1) = (K-1)(2(K-2)! +1) = 2(K-1)! + K-1.\]

Starting from index 2, we get the sequence of maximum deck sizes using this method, which is now sequence A370888:
\[3,\ 6,\ 15,\ 52,\ 245,\ 1446,\ 10087,\ 80648,\ 725769,\ 7257610,\ \ldots.\]

We see that for $K=5$, we get the standard deck size. 

Now, we give a more detailed explanation for readers who want to perform the original 5-card trick.

\textbf{Five-card trick explanation.} In the original 5-card trick, the groups are suits. Given that there are five cards total, by the pigeonhole principle, there are always two cards of the same suit. Thus, the assistant can pick one of them to hide. The other card becomes the signaling card and can be placed as the leftmost card. The remaining three cards are placed in a specific order to signal the number $S$ from 1 through 6. For this, we can assume there is an ordering on the whole deck. We can assume that the cards start with the ace of clubs, two of clubs, and so on, then move to hearts, diamonds, and spades. So, the last card is the king of spades. We assume the following lexicographic ordering of permutations: 123, 132, 213, 231, 312, and 321. The magician chooses the hidden card so that the value of the leftmost card plus $S$ equals the value of the hidden card, using a wrap-around if needed. As $S$ signals a number from 1 to 6 inclusive, we can see that with 13 cards in a suit and given two cards $A$ and $B$ of the same suit, either $A$ can be reached from $B$ while counting clockwise in 6 steps or vice versa.

\subsection{Cards can be flipped}

In this variation, we are allowed to have some cards face down, as in Mulcahy's 4-card trick \cite{Mulcahy}. Mulcahy's method here is a generalization of Cheney's method. In this method, we divide the strategy into two cases: when all the cards are face down and otherwise. By agreement, all the cards face down means a specific card. Otherwise, we use Cheney's method. We divide all the other cards in the deck into $K-1$ groups. This way, the $K$ chosen cards have at least two cards in the same group. We pick two cards from the same group and hide the one that needs a smaller signaling number $S$. The maximum signaling number $M$ is bounded by the number of ways to flip and permute the cards.

\begin{theorem}
When the cards are allowed to be face down and not allowed to be rotated, and Mulcahy's method is used, the maximum deck size is
\[N= 1 + (K-1)(2M+1) = 1+ (K-1)\left( 1+ 2\sum_{i=1}^{K-1}(i-1)!{K-1 \choose i} \right).\]
\end{theorem}

\begin{proof}
We have one case when all the cards are face down.

Suppose not all the cards are face down. Mulcahy's method is similar to Cheney's method. Here, the total number of groups is $K-1$, and the first face-up card (the signaling card) designates the group. We calculate the maximum signaling number in the following way. If we have $i$ face-up cards, we have ${K-1\choose i}$ ways to choose their places. After that, we have to place the signaling card at a particular spot and we have $(i-1)!$ ways to order the remaining face-up cards. Thus, the maximum signaling number is 
\[M = \sum_{i=1}^{K-1}(i-1)!{K-1 \choose i},\]
and the deck size is $(K-1)(2M+1)$ due to Eq.~\ref{eq:cm}. Combining both cases, we arrive at this theorem.
\end{proof}

Starting from index 1, we get a new sequence A371217 in the OEIS \cite{OEIS}
\[1,\ 4,\ 15,\ 52,\ 197,\ 896,\ 4987,\ 33216,\ 257161,\ 2262124,\ 22241671,\ \ldots. \]

For readers who want to perform the original Mulcahy's 4-card trick \cite{Mulcahy}, we give a more detailed explanation here from the assistant's point of view.

\textbf{Four-card trick explanation.} It is easier to explain this trick using numbers. Suppose we have cards with numbers from 0 to 51 inclusive. Consider 51 a special case. If 51 is among our four cards, we hide it and represent it with all the cards face down. Suppose we do not have 51 among the four cards. Then, two cards out of four with the same quotient modulo 17 exist. We denote them as $A$ and $B$. This idea is similar to the suit idea in Cheney's method. We have the flexibility to reorder $A$ and $B$, so we choose $A$ so that $S = B-A < 9 \pmod {17}$. The value $S$ becomes our signaling number. Now, we hide $B$ and place $A$ face up as the leftmost face-up card, using $A$ as a signaling card. We use the other cards to represent a signaling number between 1 and 8 inclusive. We consider face-down cards as zeros and face-up cards as ones and read them as binary. As not all the cards are face down, we can get a number 1 through 7. In addition, when all the cards are face up, we have extra flexibility; we can swap two cards that are not the signaling card. Thus, we can agree that when these two cards are in order, the signaling number is 7; otherwise, it is 8.

\subsection{Cards can be flipped and rotated}

Now, we look at a variation where each card can be rotated in $R$ ways. We also allow the cards to be flipped as before. We calculate the maximum deck size if we use Mulcahy's strategy.

\begin{theorem}
When the cards are allowed to be rotated and flipped and Mulcahy's method is used, the maximum deck size is
\[N = R^{K-1}+(K-1)\left(2R^{K-1} \sum_{i=1}^{K-1} \binom{K-1}{i}(i-1)! + 1\right).\]
\end{theorem}

\begin{proof}
If all the cards are face down, there are $R$ ways to rotate each card. This results in $R^{K-1}$ different outcomes. 

Now we assume that there are face-up cards and the leftmost face-up card is the signaling card. To calculate the maximum signaling number $M$, we do casework based on how many face-up cards there are. If there are $i$ face-up cards, then the number of different messages is a product of $R^{K-1}$ rotations, $\binom{K-1}{i}$ ways to choose where to place the face-up cards, and $(i-1)!$ ways to permute the face-up cards. Thus, summing over $i$, we get
\[M = \sum_{i=1}^{K-1} R^{K-1}\binom{K-1}{i}(i-1)!.\]
 By Eq.~\ref{eq:cm}, when not all the cards are face down, the largest deck is $(K-1)(2M+1)$. Adding $R^{K-1}$, we get the final formula.
\end{proof}

Table~\ref{tab:rotfl} shows the maximum deck sizes for small values of $K$ and $R$. For $R=1$, we get the same value as before, when we allowed the cards to be flipped but didn't allow rotations.

\begin{table}[ht!]
\begin{center}
\begin{tabular}{|c|c|c|c|c|c|c|}
\hline
$R/K$ &1&2&3&4&5&6 \\
\hline
1&1&4&15&52&197&896 \\
2&1&7&54&395&3092&28517 \\
3&1&10&119&1326&15637&216518 \\
4&1&13&210&3139&49412&912389 \\
5&1&16&327&6128&120629&2784380 \\
\hline
\end{tabular}
\end{center}
\caption{Maximum deck sizes for small values of $K$ and $R$ when the cards are in a line and can be flipped. The assistant chooses the hidden card.}
\label{tab:rotfl}
\end{table}

\begin{example}
Consider $K=3$ and $R=2$. The largest possible deck is 54, slightly bigger than the standard deck. Thus, the trick works for the standard deck. As we mentioned before, replacing rotations with an asymmetric deck is an option that will make impressive magic. We describe this new trick in the next section. 
\end{example}

\section{The new 3-card trick}
\label{sec:trick}

\begin{tcolorbox}
\textbf{The 3-card trick.} The audience gives the magician's assistant 3 cards from a standard deck. The assistant hides one card and then arranges the two other cards in a row, face up or down. In addition, each card can be placed horizontally or vertically. The magician uses the description of how the cards are placed in a line to guess the hidden card.
\end{tcolorbox}

\subsection{Preparation}

When talking about card rotations, an unrotated card represents the number 1 (because it looks like a 1), and a rotated card represents 0, as shown in Table~\ref{tab:bin}. The \textit{rotating number}, denoted $R$, is the number obtained by reading the card rotations as binary, from left to right (0 through 3). For example, if both cards are unrotated, then $R = 11_2 = 3$.

When talking about flipped and unflipped cards, a face-up card represents 0, and a face-down card represents 1, as shown in Table~\ref{tab:bin}. (We made this choice so that the count starts with zero.) The \textit{flipping number}, denoted as $F$, is the number obtained by reading the card flips as binary, from left to right. There are four possible flipping numbers, 0 through 3, but the flipping number $11_2 = 3$ (both cards face down) is reserved for a special case, which will be described below. This means there are effectively 3 flipping numbers, 0 through 2.

\begin{table}[ht!]
\begin{center}
\begin{tabular}{ |c|c|c| } 
 \hline
 & Flipping Number & Rotation Number \\
 \hline
 0 & Unflipped & Rotated \\ 
 \hline
 1 & Flipped & Unrotated \\ 
 \hline
\end{tabular}
\end{center}
\caption{From binary to orientation conversion.}
\label{tab:bin}
\end{table}

\subsection{Assistant's task}

We divide the explanation into two cases: when the assistant gets an ace and otherwise.

\textbf{Ace case:} When the assistant's hand includes an ace, the ace should be selected as the hidden card. Then, the assistant should place the other two cards face down to signal that the hidden card is an ace. The rotating number 0 represents hearts, 1 represents diamonds, 2 represents clubs, and 3 represents spades.

\textbf{No-ace case:} Suppose the assistant's hand doesn't include an ace, then, by the pigeonhole principle, there are at least two cards of the same color. The cards within a suit should be numbered from 2 to 13, as we do not need 1 because there is no ace. Let the numbers on the cards be $x$ and $y$, where $x \geq y$. At this point, there are two cases for what the assistant can do, depending on whether the two cards of the same color are in the same suit.

\begin{enumerate}
\item The same-color cards have the same suit. The assistant will hide the larger number $x$, and put $y$ in the leftmost face-up position as the signaling card. Then, the signaling number will be ${x-y}$, which is between 1 and 11.
\item The same-color cards are from different suits. This time, the assistant selects the smaller number $y$, and displays $x$ as the signaling card in the leftmost face-up position. Then, the signaling number will be $12-(x-y)$, which is between 1 and 12.
\end{enumerate}

Once the assistant knows the signaling number $S$, the assistant must solve the equation $4F+R+1=S$, where $S$ is the signaling number and $F$ and $R$ are flipping and rotating numbers defined above. A convenient way for the assistant to calculate this is by calculating $R = S-1 \pmod 4$ and then finding $F$. Recall that the assistant will have at least one face-up (unflipped) card, implying that $F$ is in the range $[0\text{--}2]$. The rotating number is in the range $[0\text{--}3]$. Thus, the signaling number is in the range $[1\text{--}12]$. Then, the assistant arranges the cards according to the values of $F$ and $R$ and positions the cards with the help of Table~\ref{tab:bin}. The signaling card, the card of the same color, must be the leftmost face-up card.

\subsection{Magician's task}

When both cards are face down, the hidden card is an ace. The conversions in Table~\ref{tab:bin} provide a value of $R$. An $R$ value of 0 represents hearts, 1 represents diamonds, 2 represents clubs, and 3 represents spades.

Suppose at least one of the cards is face up. The magician calculates the values of $R$ and $F$ using Table~\ref{tab:bin}. Then, the magician calculates the signaling number, $S=4F+R+1$. The color of the missing card is the same as that of the signaling card: the leftmost face-up card. Let the value of the signaling card be $x$. Then, there are two possibilities for how the magician must calculate the value of the hidden card.

\begin{enumerate}
\item $x+S\leq 13$. Then, the suit of the hidden card is the same as the suit of the signaling card, and the value of the hidden card is $x+S$.
\item $x+S\geq 14$. Then, the suit of the hidden card differs from the suit of the signaling card, but the color must still be the same. The value of the hidden card is $(x+S)-12$.
\end{enumerate}

\subsection{Example}

Say the assistant receives a 7 of diamonds, a queen of hearts, and a 3 of spades. Since both the hearts and diamonds are red, the assistant should convert them into their numbers, which are 12 and 7, respectively. Since the cards are of different suits, the assistant hides the smaller number, the 7 of diamonds. Then 12, or the queen of hearts, becomes the signaling card. The signaling number $S$ is $12 - (12-7) =7$, so the flipping number $F$ must be 1, and the rotating number $R$ must be 2. This means that out of the two cards the magician sees, the leftmost one must be unrotated and unflipped, while the rightmost one is rotated and flipped. The signaling card must be face up, so the magician will see an unrotated and face-up queen of hearts on the left and a rotated and flipped card on the right (which is the 3 of spades). The hidden card is the 7 of diamonds.

Now, we move on to the protocol for the magician. The magician calculates the flipping and rotating numbers, finding them to be 1 and 2, respectively, implying that the signaling number is $4\cdot 1+2+1=7$. The first card the magician sees, the queen of hearts, has a value of 12. Since $12+7=19\geq 14$, the value of the hidden card is $19-12$, or $7$. The suit must be different; thus, the hidden card is the 7 of diamonds.

\section{Cards can be rotated and flipped. The audience chooses the hidden card.}
\label{sec:rfaud}

We studied the case when the assistant chooses the hidden card above, due to the fact that it corresponds to a known 4-card trick \cite{Mulcahy}. Now, we turn our attention to the case when the audience chooses the hidden card.

\begin{theorem}
When the cards are allowed to be rotated and flipped, and the audience chooses the hidden card, the maximum deck size is bounded as:
\[N \leq R^{K-1}(K-1)!\sum\limits_{i=0}^{K-1}\frac{1}{i!}{K-1\choose i}.\]
\end{theorem}

\begin{proof}
The number of ways to rotate the cards is $R^{K-1}$ because each can be rotated in $R$ ways. Let $i$ be the number of face-down cards. The number of ways to choose $i$ face-down cards is $K-1\choose i$. The number of ways to choose where the $i$ face-down cards are placed among all the cards is $K-1\choose i$. The number of ways to order the remaining face-up cards is $(K-i-1)!$. The maximum number that can be signaled with $i$ face-down cards is ${K-1\choose i}^2(K-i-1)!$. Summing this up, we get
\[R^{K-1}\sum\limits_{i=0}^{K-1}{K-1\choose i}^2(K-i-1)! = R^{K-1}(K-1)!\sum\limits_{i=0}^{K-1}\frac{1}{i!}{K-1\choose i}.\]
\end{proof}

Table~\ref{tab:rfaud} shows the bounds for the maximum deck sizes for small values of $K$ and $R$.
\begin{table}[ht!]
\begin{center}
\begin{tabular}{|c|c|c|c|c|c|c|}
\hline
$R/K$ & 1 & 2 & 3 & 4 & 5 & 6 \\
\hline
1 & 1 & 2 & 7 & 34 & 209 & 1546 \\
2 & 1 & 4 & 28 & 272 & 3344 & 49472 \\
3 & 1 & 6 & 63 & 918 & 16929 & 375678 \\
4 & 1 & 8 & 112 & 2176 & 53504 & 1583104 \\
5 & 1 & 10 & 175 & 4250 & 130625 & 4831250 \\
\hline
\end{tabular}
\end{center}
\caption{Bounds for the maximum deck sizes for small values of $K$ and $R$ when the cards can be flipped. The audience chooses the hidden card.}
\label{tab:rfaud}
\end{table}

\begin{example}
For $K=2$, we have $N \leq R({1 \choose 0} + {1 \choose 1}) = 2R$. This bound is achievable. The only card on the table can be either face down or face up and can be rotated in $R$ ways. Thus, the number of different messages, the maximum signaling number, is $2R$.
\end{example}

For $R=1$, we get sequence A002720 in the OEIS, and starting from $K=1$, it is:
$$1,\ 2,\ 7,\ 34,\ 209,\ 1546,\ 13327,\ 130922,\ 1441729,\ 17572114,\ \ldots.$$ 

\begin{example}
For $R=1$ and $K=3$, the bound gives us $N=7$. We describe the strategy, noting that $X$ means a face-down card where all the numbers are considered modulo 7.

Table~\ref{tab:faud3} shows what the magician sees and the corresponding hidden card.

\begin{table}[ht!]
\begin{center}
\begin{tabular}{|c|c|}
\hline
\text{Message} & \text{Hidden Card} \\ \hline
$X$, $a$ & $a+1$ \\
$a$, $X$ & $a+3$ \\
$a$, $a \pm 1$ & $a+5$ \\
$a$, $a \pm 2$ & $a+4$ \\
$a$, $a \pm 3$ & $a+2$\\
\hline
\end{tabular}
\end{center}
\caption{The strategy for $K = 3$. The audience chooses the hidden card.}
\label{tab:faud3}
\end{table}

Here is the assistant's algorithm. Suppose $b$ is the card chosen by the audience. The leftover cards are $x$ and $y$. If $x$ and $y$ contain $b-1$, the assistant shows the message $X,b-1$. If they contain $b-3$, the message is $b-3,X$. Otherwise, the available cards are two cards from the set of four cards $\{b-2,b+1,b+2,b+3\}$. There are six possibilities for two cards. Here is the list in which order the cards should be displayed for each pair:
\[(b-2,b+1), \quad (b-2,b+2), \quad (b+3, b-2), \quad (b+2,b+1), \quad (b+3,b+1), \quad (b+2,b+3).\]
\end{example}

\begin{remark}
This method always has at least one excluded card. It might seem that the bound should be increased by 1. However, we do not use the message $XX$, which decreases the maximum signaling number.
\end{remark}

In the previous scenarios, when the audience chose the hidden card, we reached the information-theoretic bound. Here, for small $K$, we can reach the bound too. However, whether our luck will continue for larger $K$ is unclear. This is because our strategy is not efficient. For example, if the assistant gets the cards $a$ and $a+2$ to display with $a+3$ hidden, there are multiple ways to proceed. The assistant can display $(X,a+2)$ or $(a,X)$. This does not necessarily imply that the bound is not tight.

\section{Duplicates}
\label{sec:duplicates}

In this next variation, we look at a deck that consists of cards 1 to $\frac{N}{2}$ inclusive, but each number appears twice. We denote the number of distinct cards by $D$. That is, $D = \frac{N}{2}$. We assume that the cards are always face up.

\subsection{The audience chooses the hidden card}

\begin{theorem}
When the deck consists of duplicate cards, and the audience chooses the hidden card, the maximum deck size is
\[N = 2 \frac{(K-1)!}{2^{\floor*{ \frac{K-1}{2}}}}+ 2 \floor*{ \frac{K-1}{2}}.
\]
\end{theorem}

\begin{proof}
Our main formula in Eq~\ref{eq:fu} gives $N=2D = 2(M+E)$. If the audience chooses the hidden card, all the assistant can do with the remaining cards is permute them. The worst case scenario for the maximum signaling number is when there is the maximum amount of duplicates in the $K-1$ cards the assistant gets, which can reach $\floor{\frac{K-1}{2}}$, which gives us the maximum signaling number.
\[
M = \frac{(K-1)!}{2^{\floor*{\frac{K-1}{2}}}}.
\]

The worst-case scenario for excluded cards is when all cards are distinct, in which case we can't exclude anything. These two worst-case scenarios don't happen at the same time. When there is an extra duplicate in the set of cards, the maximum signaling number is twice less than otherwise, while the number of excluded cards increases by only 1. This means that the worst-case scenario for the sum of the maximum signaling number and the excluded cards occurs at the worst-case scenario of the maximum signaling number, which is when there are $\floor*{\frac{K-1}{2}}$ duplicate cards in the set. This would give us
\[
E=\floor*{\frac{K-1}{2}}
\]
and imply the bound. As usual, the bound is achievable.
\end{proof}

The sequence of the maximum deck size as a function of $K$ and starting from index 1 is below. It is 2 times the new sequence A372256:
\[2,\ 2,\ 4,\ 8,\ 16,\ 64,\ 186,\ 1266,\ 5048,\ 45368,\ 226810,\ 2494810,\ 14968812,\ \ldots.\]


\subsection{The assistant chooses the hidden card}

We estimate the bound when the assistant chooses the hidden card.

\begin{theorem}
When the deck consists of duplicate cards and the assistant chooses the hidden card, the number of distinct cards $D$ is bounded as:
\[\sum_{i=0}^{\lfloor K/2\rfloor} {D\choose i}{D-i\choose K-2i} \le \sum_{i=0}^{\lfloor (K-1)/2\rfloor} {D\choose i}{D-i\choose K-2i-1}\frac{(K-1)!}{2^i}.\]
\end{theorem}

\begin{proof}
First, we calculate the possible number of hands the assistant gets. If there are $i$ duplicates, we can choose which cards are duplicates in ${D\choose i}$ ways, then we can choose the leftover cards in ${D-i\choose K-2i}$ ways. Thus, the total number of hands is:
\[\sum_{i=0}^{\lfloor K/2\rfloor} {D\choose i}{D-i\choose K-2i}.\]
To calculate the total number of messages, suppose there are $i$ duplicates among the displayed cards. Then there are ${D\choose i}$ ways to choose these duplicates, there are also ${D-i\choose K-2i-1}$ ways to choose the leftover cards, and there are $\frac{(K-1)!}{2^i}$ ways to permute the cards. The total is
\[\sum_{i=0}^{\lfloor (K-1)/2\rfloor} {D\choose i}{D-i\choose K-2i-1}\frac{(K-1)!}{2^i}.\]
Thus, we get the inequality in the statement.
\end{proof}

We leave it to the reader to check that for $K=1$ and 2, the maximum value for $D$ is 1. Similarly, for $K=3$, the maximum value for $D$ is 4.

For $K = 4$, we get
\[{D \choose 4} + D {D -1\choose 2} + {D \choose 2} \le 3!{D \choose 3} + 3 D (D -1).\]
We leave it for the reader to show that it is equivalent to
\[D^2-17D-30 \le 0,\]
implying that $D\le 18$.

When $K=5$, we get an equation
\[(D-2)\left(\frac{(D-3)(D-4)}{120} + \frac{(D-3)}{6}+ \frac{1}{2}\right) \le (D-2)(D-3) + 6 (D-2) + 3,\]
which is equivalent to
\[D^3 -109D^2 -134D + 336 \le 0.\]
After solving it, we get that $D \le 110$.

Thus, starting from $K=1$, the bound gives the maximum deck size of
\[2,\ 2,\ 8,\ 36,\ 220.\]

Here, we look at strategy examples for $K=2$ and 3.

\begin{example}
Suppose $K = 2$ and $N = 4$ (two distinct cards). Then, there are only 2 possible messages the magician receives, but there are 3 possible sets of 2 cards. Therefore, there can't be 2 distinct cards in the deck. Thus, the maximum deck size is $N = 2$. 
\end{example}

\begin{example}
For $K = 3$ and $N = 8$, there are 16 possible messages the magician can receive, and there are $\binom{4}{3}+4\cdot 3=16$ possible sets of 3 cards. A strategy is shown in Table~\ref{tab:dfu3}, where we assume that cards are numbered modulo 4.

\begin{table}[ht!]
\begin{center}
\begin{tabular}{|c|c|c|}
\hline
Cards given to the assistant & Cards seen by the magician & Guess\\
\hline
$a$, $a+1$, $a+2$ & $a$, $a+1$ & $a+2$\\
$a$, $a$, $a+1$ & $a$, $a$ & $a+1$\\
$a$, $a$, $a+2$ & $a$, $a+2$ & $a$\\
$a$, $a$, $a+3$ & $a$, $a+3$ & $a$\\
\hline
\end{tabular}
\end{center}
\caption{The strategy for duplicates in case $K=3$.}
\label{tab:dfu3}
\end{table}

We can describe the assistant's strategy. If there are no duplicate cards, order them consecutively mod $4$, hide the third, and show the first two in order. This corresponds to the first line in the table. If there is a pair of duplicates, let us denote the duplicate card by $a$ and the other card by $b$. If $b \equiv a+1 \mod{4}$, hide the non-duplicate card and show the duplicates. This corresponds to the second line in the table. Otherwise, hide one of the duplicates and put the other in the first position. This corresponds to the third and fourth lines in Table~\ref{tab:dfu3}.

The magician uses Table~\ref{tab:dfu3} to guess the card.
\end{example}

We see that our examples reach our information-theoretic bound. We do not expect it to continue. However, there is a very simple general strategy which we will describe right now.

\subsection{Trick strategy: Signaling the duplicate card.}

If there are any duplicate cards among the $K$ cards, the assistant puts one of those duplicates as the leftmost card and hides the other. The assistant then arranges all the other cards in non-decreasing order. For example, if the cards the assistant gets are 6, 9, 9, 10, 10, and 11, he/she can show 9, 6, 10, 10, and 11 and hide the other 9.

Note that if there are at least two pairs of duplicates, we can choose which duplicate card to show. In the example above, we can signal 10 by presenting the cards in order 10, 6, 9, 9, 10, and 11. Moreover, if there are several duplicates, we can relax the rules. When the magician sees duplicates among displayed cards, the magician knows there are duplicates; thus, the hidden card is the duplicate of the leftmost card. So, the assistant can arrange other cards in any order.

Now, we assume all the cards given to the assistant are unique. We can use the best-trick strategy, where we are not allowed a permutation with all the cards except the first one in increasing order. Thus, there are $K-1$ forbidden permutations, implying that our maximum signaling number is 
\[M = (K-1)!-K+1.\]
As in the best trick, we multiply this number by $K$ and add excluded cards to get the number of distinct cards to be $D = KM + E = K((K-1)!-K+1)+K-1$. Thus, this method gives the maximum deck size
\[N=2(K((K-1)!-K+1)+K-1),\]
which simplifies to
\[N= 2(K!-K^2+2K-1).\]
The corresponding sequence starting from $K=1$ is
\[2,\ 2,\ 4,\ 30,\ 208,\ 1390,\ 10008,\ 80542,\ 725632,\ 7257438,\ \ldots,\]
which is now twice sequence A372264.


For $K=4$ and $K=5$, the difference between the deck size achieved in this strategy and the maximum indicated by the information-theoretical bound is quite small.

\begin{example}
For $K=3$ and $N=4$, there are two distinct cards, 0 and 1, and the assistant can get either 001 or 011. Since there has to be a duplicate, the magician's guess is solely based on the leftmost card.
\end{example}

\section{Guessing more than 1 card}
\label{sec:morethanone}

What if the magician needs to guess not one but more cards? Here, we consider the case when the magician needs to guess $C > 1$ cards. We allow the cards to be rotated but assume that they are face up.

\subsection{The audience chooses the hidden cards}

As before, when the audience chooses the hidden cards, we can calculate the maximum deck size exactly.

\begin{theorem}
When the audience chooses $C$ hidden cards, the maximum deck size satisfies the equation
\[{N-K+C \choose C} \le R^{K-C}(K-C)!.\]
\end{theorem}

\begin{proof}
The assistant can permute and rotate the available cards. Thus, the maximum signaling number $M$ is $R^{K-C}(K-C)!$. The magician needs to guess the $C$ cards that are not shown. The number of possibilities is $N-K+C \choose C$. Thus, the maximum deck size is bounded:
\[{N-K+C \choose C} \le R^{K-C}(K-C)!.\]
Similar to before, the bound is achievable. The assistant and magician agree on how to index permutations and possible sets of $C$ cards. Then, the assistant makes a permutation to match the hidden set.
\end{proof}

We separately discuss the case of two hidden cards.

\begin{corollary}
When the audience chooses the two hidden cards, the maximum deck size is
\[N = \left\lfloor \frac{2K-3 + \sqrt{1+8R^{K-2}(K-2)!}}{2} \right\rfloor.\]
\end{corollary}

\begin{proof}
Our equation, in this case, becomes
\[ N^2 -2NK + 3N -3K+ K^2+2 \leq 2R^{K-2}(K-2)!.\]
Using the quadratic formula, we can find the largest root and get our result.
\end{proof}

Table~\ref{tab:2aud} shows the bound for small values of $R$ and $K$ when two cards are hidden.
\begin{table}[ht!]
\begin{center}
\begin{tabular}{|c|c|c|c|c|c|}
\hline
$R/K$ & 2 & 3 & 4 & 5 & 6 \\
\hline
1 & 2 & 3 & 4 & 7 & 11 \\
2 & 2 & 3 & 6 & 13 & 32 \\
3 & 2 & 4 & 8 & 21 & 66 \\
4 & 2 & 4 & 10 & 31 & 115 \\
5 & 2 & 4 & 12 & 42 & 177 \\
\hline
\end{tabular}
\end{center}
\caption{Maximum deck sizes for small values of $K$ and $R$ when the cards are in a line and the audience chooses the two hidden cards.}
\label{tab:2aud}
\end{table}

The first row, when $R=1$, corresponds to the new sequence A372266 (starting from $K =2$)
\[2,\ 3,\ 4,\ 7,\ 11,\ 21,\ 44,\ 107,\ 292,\ 861,\ 2704,\ 8946,\ 30964,\ \ldots.\]

\begin{example}
For $K=5$, the deck size $N=7$. The assistant has three cards that can be permuted in six ways. Thus, the assistant can signal any number between 1 and 6, inclusive. On the other hand, there are four unseen cards. They form six pairs of cards. We can arrange the six pairs of cards in lexicographic order, making each pair correspond to a number 1 through 6. For example, suppose the assistant gets cards 2, 3, 4, 5, and 6, and the audience hides 3 and 5. Then possible pairs of unseen cards are $(1,3)$, $(1,5)$, $(1,7)$, $(3,5)$, $(3,7)$, and $(5,7)$ in lexicographic order. The pair the magician needs to guess is number four. Thus, the assistant uses the lexicographic order of permutations and picks permutation $(4,6,2)$ to signal number 4.
\end{example}

\subsection{The assistant chooses the hidden cards}

Now, we consider the case of the assistant choosing the hidden cards.

\begin{theorem}
When the assistant chooses $C$ hidden cards, the maximum deck size is bounded as
\[\prod\limits_{i=1}^{C}(N-K+i)\leq R^{K-C} K!.\]
\end{theorem}

\begin{proof}
We can calculate the information-theoretic bound the same way as before. The number of possible hands the assistant gets can't be more than the number of messages:
\[{N\choose K} \leq \frac{R^{K-C}N!}{(N-K+C)!}.\]
We can simplify this to 
\[(N-K+C)!\leq R^{K-C} K! (N-K)!\]
and
\[\prod\limits_{i=1}^{C}(N-K+i)\leq R^{K-C} K!.\]
\end{proof}

We separately discuss the case of two hidden cards.
\begin{corollary}
When the assistant chooses two hidden cards, the maximum deck size is bounded as
\[N \leq \frac{2K-3 + \sqrt{1+4R^{K-2}K!}}{2}.\]
\end{corollary}

\begin{proof}
When $C = 2$, our inequality becomes
\begin{equation}
\label{eq:2cardbound}
(N-K+2)(N-K+1)\leq R^{K-2} K!.
\end{equation}
After collecting the terms with the same powers of $N$, we get
$$N^2+N(3-2K) +(K^2-3K+2 - R^{K-2} K!)\leq 0.$$

Using the quadratic formula, we can find the largest root
\[\frac{2K-3 + \sqrt{9+4K^2-12K-4K^2+12K-8+4R^{K-2}K!}}{2},\]
which simplifies to
\[\frac{2K-3 + \sqrt{1+4R^{K-2}K!}}{2},\]
implying the statement.
\end{proof}

Table~\ref{tab:2ass} shows the bound values for small values of $R$ and $K$ when $C=2$.

\begin{table}[ht!]
\begin{center}
\begin{tabular}{|c|c|c|c|c|c|c|}
\hline
$R/K$ & 2 & 3 & 4 & 5 & 6 & 7\\
\hline
1 & 2 & 4 & 7 & 14 & 31 & 76\\
2 & 2 & 5 & 12 & 34 & 111 & 407\\
3 & 2 & 5 & 17 & 60 & 245 & 1112\\
4 & 2 & 6 & 22 & 91 & 433 & 2277\\
5 & 2 & 7 & 27 & 125 & 675 & 3974\\
\hline
\end{tabular}
\end{center}
\caption{The bounds for the deck sizes for small values of $K$ and $R$ when the cards are in a line and the assistant chooses the two hidden cards.}
\label{tab:2ass}
\end{table}

The first row corresponds to new sequence A372265
\[0,\ 2,\ 4,\ 7,\ 14,\ 31,\ 76,\ 207,\ 609,\ 1913,\ 6327,\ 21896,\ 78922,\ 295272,\ 1143549,\ \ldots.\]

\begin{example}
For $K=3$ and $N=4$, the strategy is as follows. The assistant gets three cards out of four. Suppose card $a$ is missing. Then, the assistant displays $a+1$ modulo 4, and the magician guesses $a+2$ and $a+3$ modulo 4. In other words, we can assume that the deck is arranged in a circle. The assistant gets three consecutive cards and shows the first card. The magician guesses the next two cards.
\end{example}

\begin{example}
Suppose $K=4$ and $N=7$. Considering the 4 cards modulo 7, we have 5 cases up to clockwise rotation: a) cards are consecutive, b) three consecutive cards, then skip one, c) three consecutive cards, then skip two, d) two pairs of consecutive cards, e) one pair of consecutive cards and two isolated cards. These cases and the assistant's strategy are presented in Table~\ref{tab:Guessing2:4-7}.

\begin{table}[ht!]
\centering
\begin{tabular}{|c|c|}
\hline
 Message	& Hidden cards 	\\ \hline 
$a,\ a+1$	& $a+2,\ a+3$		\\  
$a,\ a-1$	& $a+1,\ a+3$		\\  
$a,\ a+2$	& $a+1,\ a-2$		\\ 
$a,\ a+3$	& $a-3,\ a-1$		\\  
$a,\ a-2$	& $a+1,\ a+3$		\\ \hline 
\end{tabular}
\caption{A strategy for guessing two cards when $K=4$ and $N=7$.}
\label{tab:Guessing2:4-7}
\end{table}
\end{example}

\subsection{The best-card-trick method reused for two hidden cards}

The assistant gets $K$ cards, and at first, the assistant pretends they need to hide only one card $a$. As in the best trick, card $a$ is selected as card $c_i$, where $i$ is the sum of the $K$ cards mod $K$. Then, cards are re-indexed, and card $b$ is selected as card $c_j$, where $j$ is the sum of the remaining cards mod $K-1$. The magician guesses the cards in order $b$, then $a$.

When the magician guesses card $b$, there are $N-E = N-K+2$ possibilities for it. The fact that the best-trick method was used to choose $b$ means that the magician can find the value of $b$ mod $K-1$. Thus, there are up to $\lceil \frac{N-K+2}{K-1}\rceil$ possibilities for the value of $b$, which the magician needs to distinguish. Thus, the maximum signaling number for $b$ should be $\lceil \frac{N-K+2}{K-1}\rceil$. Similarly, the maximum signaling number for $a$ should be $\lceil \frac{N-K+1}{K}\rceil$.

The total number the assistant can signal is $R^{K-2}(K-2)!$, which can't be less than the product of the numbers above. Thus, we have
\[
\left\lceil \frac{N-K+1}{K}\right\rceil \left\lceil \frac{N-K+2}{K-1}\right\rceil
\le R^{K-2}(K-2)!.
\]
This is almost exactly the formula for the above-mentioned bound in Eq.~\ref{eq:2cardbound}. Table~\ref{tab:2assbctm} shows the maximum deck size corresponding to the strategy we describe for small values of $R$ and $K$ when $C=2$. We see that the difference between the bound in Table~\ref{tab:2ass} and the achievable deck size in Table~\ref{tab:2assbctm} is minor.

\begin{table}[ht!]
\begin{center}
\begin{tabular}{|c|c|c|c|c|c|c|} 
\hline $R/K$ & 2 & 3 & 4 & 5 & 6 & 7 \\ 
\hline
1   & 1 & 3 & 7 & 14 & 29  & 76  \\
2   & 1 & 5 & 11 & 34 & 109 & 405 \\
3   & 1 & 5 & 15 & 59 & 244 & 1109 \\
4   & 1 & 5 & 20 & 89 & 431 & 2274 \\
5   & 1 & 5 & 26 & 124 & 671 & 3971 \\
\hline 
\end{tabular}
\end{center}
\caption{Maximum deck sizes for small values of $K$ and $R$ when two cards are hidden and the best-card-trick method is used.}
\label{tab:2assbctm}
\end{table}

This strategy can be extended to the case when more than two cards are hidden.

\begin{example}
When $R=1$ and $K=5$, our bound for the strategy gives: $\left\lceil \frac{N-4}{5}\right\rceil \left\lceil \frac{N-3}{4} \right\rceil \le 6$, implying that $N\le 14$, which matches our theoretical bound. Now, we design a strategy. First, we see that we need to split the $(K-2)! = 6$ of permutations into $\lceil\frac{N-4}{5}\rceil=2$ for $a$ and $\lceil\frac{N-3}{4}\rceil=3$ for $b$. We can label permutations in lexicographic order and correspond them to two signaling numbers $(1,1)$, $(1,2)$, $(2,1)$, $(2,2)$, $(3,1)$ and $(3,2)$ in this order, where the first number signals $b$ and the second number signals $a$.

Say the assistant receives the cards 2, 3, 7, 9, and 13. The assistant calculates that $2+3+7+9+13 \equiv 4\mod 5,$ so the hidden card $a$ is the fifth card, which is 13. After renumbering, the card 13 becomes the card 9, which has a quotient of 1 when divided by 5, so the corresponding signaling number is 2. The assistant now has 2, 3, 7, 9, and calculates $2+3+7+9\equiv 1\mod 4$, so the hidden card $b$ is the second card, which is 3, with the signaling number 1. Thus, the assistant hides 3 and 13 while keeping 2, 7, and 9 and needs to signal $(1,2)$ pair, which corresponds to the second permutation in lexicographic order, which is $(2,9,7)$.

After the magician sees $(2,9,7)$, the realization is that the pair of signaling numbers is $(1,2)$. The hidden card $b$ could be $3$, $6$, or $11$. The actual value of $b$ is the smallest possibility, which is 3. Now, the magician knows cards 2, 3, 7, and 9. The possible hidden cards $a$ from this knowledge are $6$ or $13$. As the signaling number is 2, the actual value is 13.
\end{example}

\section{Acknowledgments}

We are grateful to the MIT PRIMES STEP program for allowing us to conduct this research.

\end{document}